\documentclass[12pt]{article}
\usepackage[english]{babel}
\usepackage[latin1]{inputenc}
\usepackage{amsfonts,amssymb,amsmath, epsfig}
\usepackage{color,graphicx,graphics,psfrag}
\usepackage{amsmath,amstext,amssymb,amsfonts, amscd}
\usepackage{amsthm}
\usepackage{tikz}

\usepackage{enumerate}

\newcommand{\dd}{\,\mathrm{d}}
\newcommand{\e}{\varepsilon}

\newtheorem{mydef}{Definition}
\newtheorem{mythm}{Theorem}
\newtheorem{myprop}{Proposition}
\newtheorem{mylem}{Lemma}
\newtheorem{myrem}{Remark}

\begin{document}

\title{Decay to equilibrium of the filament end density along the leading edge of the lamellipodium} 
\author{Angelika Manhart$^1$, Christian Schmeiser$^2$} 
\date{} 

\maketitle

\begin{center}
1. Faculty of Mathematics, University of Vienna, \\
Oskar-Morgenstern Platz 1,1090 Vienna, Austria,  \\
angelika.manhart@univie.ac.at
\end{center}

\begin{center}
2. Faculty of Mathematics, University of Vienna, \\
Oskar-Morgenstern Platz 1,1090 Vienna, Austria,  \\
christian.schmeiser@univie.ac.at
\end{center}

\vspace{0.5 cm}
\begin{abstract}
A model for the dynamics of actin filament ends along the leading edge of the lamellipodium is analyzed.
It contains accounts of nucleation by branching, of deactivation by capping, and of lateral flow along the leading edge 
by polymerization. A nonlinearity arises from a Michaelis-Menten type modeling of the branching process.
For branching rates large enough compared to capping rates, the existence and stability of nontrivial steady states
is investigated. The main result is exponential convergence to nontrivial steady states, proven by investigating the
decay of an appropriate Lyapunov functional. 
\end{abstract}

\medskip
\noindent
{\bf Acknowledgements:} This work has been supported by the Austrian Science Fund (FWF) through the doctoral school {\em Dissipation and Dispersion in Nonlinear PDEs} (project W1245) as well as the Vienna Science and Technology Fund (WWTF) (project LS13/029). The authors are grateful to J. Vic Small for many
hours of discussions and for the permission to use the drawing in Fig. \ref{fig_le:lateral-flow}.

\medskip
\noindent
{\bf Key words: } lamellipodium, actin, Lyapunov function

\medskip
\noindent
{\bf AMS Subject classification: } 
                                                   35L50, 
                                                   35Q92, 
                                                   92C37 
\vskip 0.4cm

\section{Introduction}

The lamellipodium is a thin protrusion, developing when biological cells spread on flat surfaces. It is supported by a roughly
two-dimensional meshwork of protein filaments, created by polymerization of actin \cite{Small2002}. In steadily protruding lamellipodia, the meshwork exhibits two dominant directions, approximately symmetric to the leading edge of the lamellipodium
\cite{Vinzenz}, and can thus be approximated by two distinct families of filaments. The meshwork is a very dynamic structure, driven by the 
{\em polymerization} of the filaments abutting the leading edge, but also by the nucleation of new filaments via 
{\em branching} away from old filaments, close to their growing ends. This is responsible for the two-direction structure with the angle
between the two families approximately equal to the branching angle. Finally, {\em capping} of filaments plays a role, whence
filaments become deactivated, stop to polymerize, and subsequentially loose contact to the leading edge.

This work deals with a model for the dynamics of filament ends along the leading edge. New filament ends are produced by
branching, they disappear from the leading edge by capping, and they move along the leading edge by what is called
{\em lateral flow,} a consequence of polymerization and of the inclination of filaments relative to the leading edge \cite{Small1994}
(see Fig. \ref{fig_le:lateral-flow}). Under the idealizing assumption of a constant angle between filaments and leading edge
and of a constant polymerization speed, the speed of lateral flow along the leading edge is constant. In reality, however,
this cannot be expected, since the polymerization speed is subject to various influences such as chemical signaling and
mechanical restrictions due to the varying geometry of the leading edge, where the latter will also lead to varying angles
between filaments and leading edge.

\begin{figure} \centering
\includegraphics[width=0.5\textwidth]{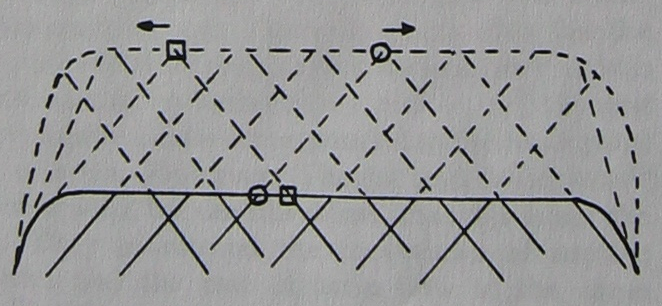}
\caption{Lateral flow. Solid and dashed lines represent the present and, respectively, a future state of filaments of 
the leading edge (drawing courtesy of J. Vic Small).}
\label{fig_le:lateral-flow}
\end{figure}

A complete model therefore needs to describe the positions of filaments and of the leading edge. The authors have been
involved in the formulation of such a modeling framework, the Filament Based Lamellipodium Model (FBLM)
\cite{MOSS,Schmeiser2010}, including accounts for filament bending, cross-linking, and adhesion to the substrate, as well as 
a number of other relevant 
mechanisms. The present work is concerned with a submodel describing branching and capping, and where the 
lateral flow speed along the leading edge will be considered as given. As a further model simplification, 
the lateral flow speed of both filament families will be assumed equal at each point on the leading edge.
Concerning the geometry, two different situations will be considered: For cells surrounded by a lamellipodium, 
the leading edge is described as a one-dimensional interval with a periodicity assumption, where the two ends are 
identified. On the other hand, for cells like the moving fish keratocyte with a crescent-like shape and a lamellipodium 
only along the outer rim, the leading edge is represented by an interval with zero lateral inflow.

The branching process requires the Arp2/3 protein complex connecting the old and the new filament at the branch
point \cite{Svitkina1999}. The assumptions that the availability of Arp2/3 is limiting and that the Arp2/3 dynamics is fast compared to the
branching dynamics results in a Michaelis-Menten type model, similar to the one already formulated in \cite{Grimm2003}.
Capping is described as a simple Poisson process. 

The main question of this work is: Does the mathematical model describe a stable distribution of filament ends?
The answer is a conditional yes with the rather obvious condition that the branching rate has to be big enough compared to the
capping rate. Otherwise the filament end population dies out.

The rest of this paper is structured as follows: In Section \ref{sec_le:derivation} the derivation of the model is described. This has already been explained in the context of the full FBLM in \cite{MOSS}, but it is also included here for the sake of completeness. In Section \ref{sec_le:general_remarks} an existence, uniqueness, and boundedness result is proven. It is also shown that it is
enough to initially have a small amount of filament ends of only one family, to make the densities of both families positive
everywhere within finite time. The short Section \ref{sec_le:Cap-gt_Bra} is concerned with the proof of the simple result that
the end distributions converge to zero, when the branching rate is too small compared to the capping rate.
In Section \ref{sec_le:stationary} existence results for non-trivial stationary states are proven. There are several kinds of 
results. First, it is shown that a transcritical bifurcation away from the zero steady state occurs, when the ratio between
the branching rate and the capping rate exceeds a critical value. This local result is extended in two special situations:
In the case of the periodic leading edge, existence of a nontrivial steady state is proven also far from the bifurcation
point, if the lateral flow speed is almost constant. The same result holds for the mathematically more difficult case of a
leading edge with zero lateral inflow, if the lateral flow speed is constant.
Finally in Section \ref{sec_le:lyapunov} it is shown that for every nonvanishing initial distribution the solution converges
exponentially to the nontrivial steady state, if it exists.

\section{Derivation and Nondimensionalisation of the Model}
\label{sec_le:derivation}

A model very similar to the one considered here has been formulated in \cite{Grimm2003}. We shall follow the derivation
given in \cite{MOSS} in the framework of the FBLM.

The leading edge is assumed as a (potentially closed) rectifiable curve of length $L$, parametrized by arclength
$x\in [0, L]$. We distinguish between two families of filaments, those pointing to the right with number density of ends $u(x,t)$ and those pointing to the left with density $v(x,t)$. By lateral flow the right-pointing filament ends are moved to the right and the left-pointing filament ends to the left, both with the prescribed position dependent speed $c(x)> 0$. It is a simplifying
assumption that the speed is time independent and the same for both families. The density of Arp2/3 is denoted by $a(x,t)$.
It is assumed to be recruited from the cytoplasm to the leading edge with a constant rate $c_\text{rec}$ with the opposite reaction
working towards the equilibrium value $a_0$. Furthermore Arp2/3 is consumed by branching events, where filament ends of
one family create ends of the other with rate constant $\kappa_\text{br}$. Arp2/3 molecules at the leading edge are assumed
immobile for simplicity. Finally, the rate constant for the deactivating capping reaction is denoted by $\kappa_\text{cap}$.
These assumptions lead to the system
\begin{eqnarray*}
\partial_t u+\partial_x (c(x)u) &=& \kappa_\text{br} \frac{a}{a_0}v-\kappa_\text{cap}u \,,\\
\partial_t v-\partial_x (c(x)v) &=& \kappa_\text{br} \frac{a}{a_0}u - \kappa_\text{cap}v \,,\\
\partial_t a &=& c_\text{rec} \left( 1 - \frac{a}{a_0}\right) - \kappa_\text{br} \frac{a}{a_0} (u+v) \,.
\end{eqnarray*}
We introduce the scaling
\begin{align*}
&x\rightarrow x\,L,\quad  t\rightarrow \frac{t}{\kappa_\text{cap}},\quad  c(x)\rightarrow L\,\kappa_\text{cap} c(x), \quad \left(u,v\right)\rightarrow \left(u\frac{c_\text{rec}}{\kappa_\text{br}}, v\frac{c_\text{rec}}{\kappa_\text{br}}\right), \quad
 a \to a_0a,
\end{align*}
and the dimensionless parameters
\begin{align*}
\alpha:=\frac{\kappa_\text{br}}{\kappa_\text{cap}} \,,\qquad \varepsilon = \frac{\kappa_\text{cap} a_0}{c_\text{rec}} \,,
\end{align*}
where $\alpha$, the ratio between the branching and the capping rates, is assumed of moderate size, whereas $\varepsilon$, the
ratio between the characteristic time for Arp2/3 and that of the capping and branching processes will be assumed as small.
Whereas the first assumption is justified and actually necessary, as our analysis will show, the smallness of $\varepsilon$
has, to the knowledge of the authors, not been verified experimentally. The nondimensionalized system has the form
\begin{eqnarray*}
\partial_t u+\partial_x (c(x)u) &=& \alpha av - u \,,\\
\partial_t v-\partial_x (c(x)v) &=& \alpha au - v \,,\\
\varepsilon\partial_t a &=& 1 - a(1+u+v) \,.
\end{eqnarray*}
The last step in the model derivation is to pass to the quasistationary limit $\varepsilon\to 0$ in the equation for $a$, which
is analogous to the derivation of Michaelis-Menten kinetics. Elimination of $a$ from the resulting system gives
\begin{align}
&\partial_t u + \partial_x \left(c(x) u\right)=\frac{\alpha\, v}{1+u+v}-u \,,\nonumber\\
&\partial_t v - \partial_x \left(c(x) v\right)=\frac{\alpha\, u}{1+u+v}-v \,,\label{eqn_le:main}
\end{align}
for $x\in[0,1]$. Two types of boundary conditions are biologically relevant. In the case of a ring-shaped lamellipodium around the whole cell we assume periodic boundary conditions. On the other hand, if we consider only a lamellipodium at the front, it is reasonable to assume that no left-moving filaments enter from the right and vice versa. These considerations allow to complement \eqref{eqn_le:main} with one of the following sets of boundary conditions:
\begin{align}
&\text{(DBC)} \quad u(0,t)=0, \quad v(1,t)=0, \quad \text{for } t > 0 \,,\label{eqn_le:DBC}\\
&\text{(PBC)}\quad u(0,t)=u(1,t), \quad v(0,t)=v(1,t), \quad \text{for } t > 0 \label{eqn_le:PBC}.
\end{align}
Throughout this paper, we will use the abbreviations $(DBC)$ and $(PBC)$ for Dirichlet Boundary Conditions and
Periodic Boundary Conditions, respectively. For $(PBC)$ we implicitly assume that also the lateral flow speed $c(x)$ is periodic.
To complete the definition of the problem, we pose initial conditions
\begin{align}
\label{eqn_le:IC}
u(x,0)=u_0(x) \,, \quad v(x,0)=v_0(x) \,,\qquad\text{for } x\in[0,1] \,,
\end{align}
with given $u_0(x), v_0(x)$, which are assumed to be non-negative and to satisfy the boundary conditions.

\section{Existence, Uniqueness, and Positivity of Solutions}
\label{sec_le:general_remarks}

We start with a reformulation of the problem, which will be useful in most of our proofs. It is based on the assumption
that the lateral flow speed is not only positive, but bounded away from zero: There exist positive constants 
$\underline{c},\overline{c}$, such that
\begin{align}
\label{eqn_le:cbound}
&0 < \underline c\leq c\leq \overline c \qquad \mbox{in } [0,1] \,.
\end{align}
An average inverse speed can then be defined by
$$
   \frac{1}{C} = \int_0^1 \frac{\dd x}{c(x)} \,.
$$
Now we introduce the transformation $(x,u,v) \leftrightarrow (X,U,V)$ by
\begin{equation} \label{eqn_le:transformation}
  X = C\int_0^x \frac{\dd\xi}{c(\xi)} \in [0,1]\,,\qquad cu = CU \,,\qquad cv = CV \,.
\end{equation}
The transformed version of \eqref{eqn_le:main} reads
\begin{align}
&\partial_t U + C\partial_X U = R_{U,V}-U \,,\nonumber\\
&\partial_t V - C \partial_X V = R_{V,U}-V \,,\label{eqn_le:Main}
\end{align}
with 
$$
  R_{U,V}(X,t) = \frac{\alpha\, V(X,t)}{1+\beta(X)(U(X,t)+V(X,t))} \,,
$$
and $\beta = C/c$, bounded from above and below by
\begin{align}
\label{eqn_le:betabound}
&0< \underline \beta := C/\overline c \leq \beta\leq C/\underline c =: \overline\beta\qquad \mbox{in } [0,1] \,.
\end{align}
Note that $(U,V)$ satisfies the same boundary conditions \eqref{eqn_le:DBC}, \eqref{eqn_le:PBC} as $(u,v)$. The transformed initial conditions read
\begin{align}
\label{eqn_le:IC-new}
U(t=0)=U_0:=\frac{c u_0}{C} \,, \quad V(t=0)=V_0:=\frac{c v_0}{C} \,,\qquad\text{in } [0,1] \,.
\end{align}
The reformulation has two effects: First, the solution of the equations by the method of characteristics is simplified, since the
transformed lateral flow speed is constant, and, second, linearization around the zero solution gives a problem with constant
coefficients. Applying the method of characteristics leads to a mild formulation.

\begin{mydef}
\label{def_le:mild}
(PBC) Let $(U,V)\in \mathcal{C}(\mathbb{T}^1\times [0,\infty))^2$ (with the torus $\mathbb{T}^1$ represented by the interval $[0,1]$) satisfy
\begin{eqnarray}
   U(X,t) &=& U_0(X-Ct)e^{-t} + \int_0^t e^{s-t} R_{U,V}(X+C(s-t),s)\dd s \,,\nonumber\\
   V(X,t) &=& V_0(X+Ct)e^{-t} + \int_0^t e^{s-t} R_{V,U}(X-C(s-t),s)\dd s \,,\label{eqn_le:mild-PBC}
\end{eqnarray}
for $(X,t)\in \mathbb{T}^1\times [0,\infty)$. Then 
\begin{equation}\label{eqn_le:uv-from-UV}
  (u(x,t),v(x,t)) = \frac{C}{c(x)} (U(X(x),t),V(X(x),t))
\end{equation}
is called a global mild solution of the problem \eqref{eqn_le:main}, \eqref{eqn_le:PBC}, \eqref{eqn_le:IC}. \\
(DBC) Let $(U,V)\in \mathcal{C}([0,1]\times [0,\infty))^2$ satisfy
\begin{eqnarray}
   U(X,t) &=& U_0(X-Ct)H(X-Ct)e^{-t} + \int_{(t-X/C)_+}^t e^{s-t} R_{U,V}(X+C(s-t),s)\dd s \,,\label{eqn_le:mild-DBC}\\
   V(X,t) &=& V_0(X+Ct)H(1-X-Ct)e^{-t} + \int_{(t-(1-X)/C)_+}^t e^{s-t} R_{V,U}(X-C(s-t),s)\dd s \,,\nonumber
\end{eqnarray}
for $(X,t)\in [0,1]\times [0,\infty)$, where $H$ denotes the Heavyside function. Then $(u,v)$ defined by \eqref{eqn_le:uv-from-UV}
is called a global mild solution of the problem \eqref{eqn_le:main}, \eqref{eqn_le:DBC}, \eqref{eqn_le:IC}. 
\end{mydef}
 
\begin{myprop}
\label{prop_le:IVP}
Let $u_0, v_0, c\in \mathcal{C}([0,1])$ satisfy \eqref{eqn_le:cbound}, $u_0,v_0\ge 0$, and the boundary conditions (DBC) or (PBC).
Then the problem \eqref{eqn_le:main}, \eqref{eqn_le:DBC}, \eqref{eqn_le:IC}, or, respectively, \eqref{eqn_le:main}, \eqref{eqn_le:PBC},
\eqref{eqn_le:IC}, has a unique, global mild solution $(u,v)$, satisfying
\begin{equation} \label{eqn_le:uv-bdd}
  0\le u(x,t) \le \frac{\overline c}{\underline c} \max\left\{ \max_{[0,1]} u_0, \alpha \right\} \,,\quad 
  0\le v(x,t) \le \frac{\overline c}{\underline c} \max\left\{ \max_{[0,1]} v_0, \alpha \right\}\,.
\end{equation}
\end{myprop}

\begin{myrem}
It is straightforward to show by differentiation of the equations that for smooth initial data $u_0,v_0\in \mathcal{C}^\infty(\mathbb{T}^1)$ for (PBC), and $u_0\in \mathcal{C}^\infty_0((0,1])$, $v_0\in \mathcal{C}^\infty_0([0,1))$ for  (DBC), and for smooth lateral flow speed $c\in \mathcal{C}^\infty([0,1])$, the solution satisfies $u,v\in \mathcal{C}^\infty([0,1]\times[0,\infty))$.
Some results of the following sections are based on computations with the strong forms \eqref{eqn_le:main} or \eqref{eqn_le:Main} 
of the differential equations. These can be justified by uniform smooth approximations of $u_0,v_0,c$, and subsequent removal
of the smoothing.
\end{myrem}

\begin{proof}
The (obvious) non-negativity and Lipschitz continuity of $R_{U,V}$ in terms of $(U,V)\in [0,\infty)^2$, as well as the bound
$R_{U,V}\le \alpha/\underline\beta$ will be sufficient for carrying out the proof.

With the Lipschitz continuity, it is straightforward to show that the right hand sides of \eqref{eqn_le:mild-PBC} and \eqref{eqn_le:mild-DBC}
preserve non-negativity and are contractions on $\mathcal{C}(\mathbb{T}\times[0,T])$ and, respectively, 
$\mathcal{C}([0,1]\times[0,T])$, for $T$ small enough, which proves local existence.
The estimate 
$$
  0\le u(x,t) \le \frac{C}{\underline c}\left( e^{-t} \sup_{[0,1]} U_0 + \frac{\alpha}{\underline\beta} (1-e^{-t})\right)
  \le \frac{C}{\underline c}\max\left\{\sup_{[0,1]} U_0, \frac{\alpha}{\underline\beta} \right\}
  \le \frac{\overline c}{\underline c} \max\left\{ \max_{[0,1]} u_0, \alpha \right\} 
$$
and the analogous version for $v$ allow to continue the local solution indefinitely and also prove \eqref{eqn_le:uv-bdd}.
\end{proof}

One expects that if the initial conditions $(u_0(x), v_0(x))$ are positive (except at the boundaries for $(DBC)$),  the same holds for the solution $(u(x,t), v(x,t))$ for all $t\geq 0$. In fact a much stronger result is true: It is enough to have positivity of the initial data on some interval for only one family; after finite time both families will be positive everywhere (except at the boundaries for $(DBC)$). The reason for this is that, although the initial mass is reduced by capping while it is transported across the domain, it remains positive. This mass, however, will trigger the creation of new filaments of the other family through the branching term. This new mass will be transported in the opposite direction and will itself cause the creation of mass of the first family. As long as the lateral flow speed $c(x)$ is bounded from below, this process happens in finite time.

\begin{myprop}
\label{prop_le:positivity}
Let the assumptions of Proposition \ref{prop_le:IVP}  hold and let the initial data satisfy
\begin{equation} \label{eqn_le:IC-ass}
  \int_0^1(u_0+v_0)\dd x > 0 \,.
\end{equation}
Then, for every $T>2\int_0^1 \dd x/c(x)$, 
\begin{align*}
&\mbox{for (PBC):}\quad u,v>0\quad \text{in }  [0,1]\times [T,\infty) \,,\\
&\mbox{for (DBC):}\quad u>0\quad \text{in } (0,1]\times [T,\infty) \,,\quad
 v>0\quad \text{in }  [0,1)\times [T,\infty) \,.
\end{align*}
\end{myprop}

\begin{figure}[h,t]
  \centering

\begin{tikzpicture}

\draw [fill=gray!20!white] (6,0) -- (8,0) -- (10,0.5) -- (10,1)--(6,0);
\draw (10.5, 0.57) node {$\mathcal{S}_0$};
\draw[->] (10.2, 0.62) to[in=20,out=150] (9.3,0.57);

\draw [ fill=lightgray] (0,1.5) -- (6,0) -- (10,1) -- (0,3.5)--(0,1.5);
\draw (0.35,3) node {$\mathcal{S}_1$};

\draw [fill=gray] (0,3.5) -- (10,1) -- (10,6) -- (0,3.5);
\draw (9.7,5.5) node {$\mathcal{S}_2$};

\draw[->] (0,0) -- (11,0) node[anchor=north] {$X$};
\draw[->] (0,0) -- (0,7) node[anchor=east] {$t$};
\draw (10,0) -- (10,7);
\draw	(0,0) node[anchor=north] {0}
		(10,0) node[anchor=north] {1};

\draw (10, 0) -- (10, -0.1);

\draw (6,-0.3) node {$X_1$};
\draw (8,-0.3) node {$X_2$};

\draw[dotted] (0,3.5) -- (10, 3.5);
\draw (-0.3,3.5) node {$T_*$};











\end{tikzpicture}

\caption{Illustration of the proof of Proposition \ref{prop_le:positivity}.}
\label{fig_le:pos}
\end{figure}
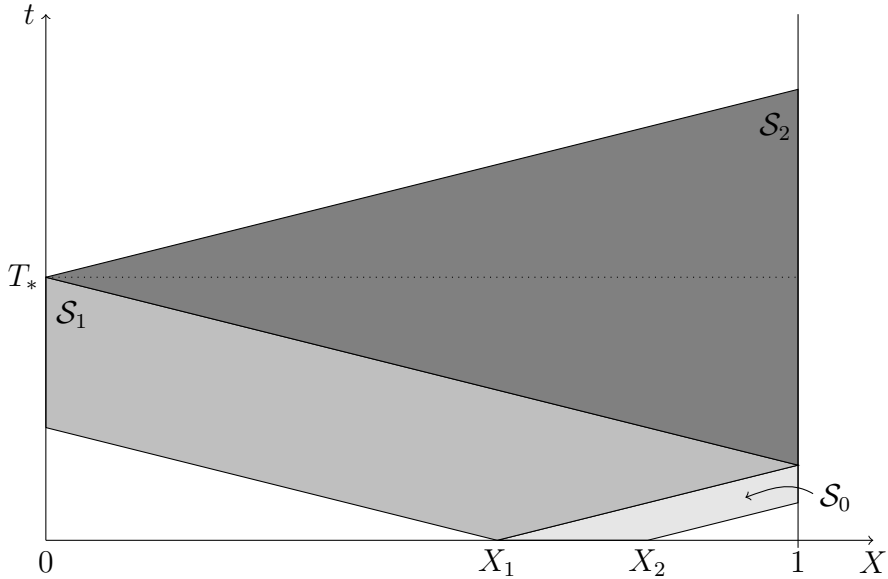

\begin{proof}
The equivalent result for $(U,V)$ will be proved. 
Assumption \eqref{eqn_le:IC-ass} and the continuity of the initial data imply the 
existence of an interval $[X_1,X_2]\subset[0,1]$ of positive length and of $m>0$, such that one of the initial densities, 
w.l.o.g. $U_0$, satisfies $U_0 \ge m$ in $[X_1,X_2]$.

The mild formulations given in Def. \ref{def_le:mild} show that, whenever a $U$-characteristic (with velocity $C$) passes through a region where $V$ is positive, 
$U$ is positive along this characteristic in and after (timewise) this region, and vice versa ($U\leftrightarrow V$, 
$C\leftrightarrow -C$). 
Also $U$ is positive everywhere along a $U$-characteristic after a point on the characteristic, where it is positive (again the same 
for $V$). 

These observations reduce the proof to a geometric problem (see Fig. \ref{fig_le:pos}), where we alternate between using the
equations for $U$ and for $V$.
The first step is the observation that by the above property of the initial data, $U$ is positive in the strip $\mathcal{S}_0$, defined by
$X_1 < X-Ct < X_2$ (light gray shading in Fig. \ref{fig_le:pos}). This implies, as the second step, that $V$ is positive there and along all 
$V$-characteristics starting in this strip, i.e. also in the region $\mathcal{S}_1$ bounded by $X=0$, $X-Ct=X_1$, $X+Ct=X_1$,
$X+Ct=2-X_1$ (gray shading in Fig. \ref{fig_le:pos}). The third step is to draw again $U$ characteristics, now starting in
$\mathcal{S}_1$, which adds a triangle $\mathcal{S}_2$ (dark gray shading in Fig. \ref{fig_le:pos}) above $\mathcal{S}_1$, where $U>0$. Straightforward continuation
shows that for $t\ge T_*=(1-X_1)/C+1/C$, $U$ and $V$ are positive for $0<X<1$. Positivity on the boundary, except 
for $U(0,t)=V(1,t)=0$ in the (DBC) case, is achieved at any time after $T_*$, implying that the result of the proposition holds
with any $T>2/C\ge T_*$.
\end{proof}

\section{When Capping Exceeds Branching}\label{sec_le:Cap-gt_Bra}

The dimensionless parameter $\alpha$ is the ratio between the branching and the capping rate. If it is too small, it can be expected that $u$ and $v$ tend to zero as $t\rightarrow \infty$. 
Note that, on the other hand, Proposition \ref{prop_le:positivity} holds for all $\alpha>0$.  
This means that even if we start with initial conditions, which are zero everywhere and are only positive for one family in a small interval, the solutions first become positive everywhere before they decay to zero. 

\begin{myprop}
\label{prop_le:convergence_small_a}
Let the assumptions of Proposition \ref{prop_le:IVP} and $0\le\alpha<1$ hold. The the solution $(u, v)$ of \eqref{eqn_le:main},
\eqref{eqn_le:DBC}, \eqref{eqn_le:IC} or \eqref{eqn_le:main}, \eqref{eqn_le:PBC}, \eqref{eqn_le:IC} satisfy 
\begin{align*}
  \|u(\cdot,t)\|_{L^2((0,1))} + \|v(\cdot,t)\|_{L^2((0,1))} \le e^{(\alpha-1)t} \left( \|u_0\|_{L^2((0,1))} + \|v_0\|_{L^2((0,1))}\right) \,.
\end{align*}
\end{myprop}

\begin{proof}
We observe that
\begin{align*}
\frac{\dd}{\dd t}\int_0^1 \frac{1}{2}\left(u^2+v^2\right)\dd x=&\int_0^1 \left( \frac{2\alpha uv}{1+u+v} - u^2 - v^2\right)\dd x\\
&+\begin{cases}
  0  \quad & \text{(PBC)} \,, \\
  -c(1)u(1,t)-c(0)v(0,t)    \quad &\text{(DBC)}\,.
  \end{cases}
\end{align*}
The non-positivity of the boundary term and the estimate
\begin{eqnarray*}
\frac{2\alpha uv}{1+u+v} - u^2 - v^2 &=& \frac{(\alpha-1)(u^2+v^2) - \alpha(u-v)^2 - (u+v)(u^2+v^2)}{1+u+v} \\
 &\le& (\alpha-1) \left(u^2+v^2\right)
\end{eqnarray*}
finish the proof.
\end{proof}

The results of the following sections show that the bound on $\alpha$ is sharp for (PBC), but not for (DBC), where decay to zero
can be expected also for $1\le \alpha< \alpha_0$ with the bifurcation value $\alpha_0$.

\section{Existence of Nontrivial Steady States}
\label{sec_le:stationary}

\subsection*{Bifurcation from the zero solution}

A bifurcation value $\alpha_0\ge 1$ for the ratio $\alpha$ between the branching rate and the capping rate will be computed,
where a nontrivial steady state branches off the zero solution. The above inequality is derived easily by repeating the 
computation in the proof of Proposition \ref{prop_le:convergence_small_a} for the steady state problem, proving that for $\alpha<1$ it has only the zero solution.

We shall work again with the variables $(X,U,V)$, where the linearization around the zero solution has constant coefficients. We
rewrite \eqref{eqn_le:Main} as
\begin{equation}\label{eqn_le:linearization}
  \partial_t \begin{pmatrix} U\\V\end{pmatrix} = \mathcal{L}_\alpha (U,V) - \frac{\alpha\beta(U+V)}{1+\beta(U+V)}
  \begin{pmatrix} V\\U\end{pmatrix} \,,
\end{equation}
with the linearized operator
$$
 \mathcal{L}_\alpha (U,V)  = \begin{pmatrix} \mathcal{L}_{\alpha,1} (U,V)\\\mathcal{L}_{\alpha,2} (U,V)\end{pmatrix} 
  = \begin{pmatrix} \alpha V - U - C\partial_X U\\\alpha U - V + C\partial_X V\end{pmatrix} \,.
$$
The bifurcation point will be the smallest value $\alpha_0\ge 1$, where $\mathcal{L}_{\alpha_0}$, subject to the boundary conditions
\eqref{eqn_le:DBC} or \eqref{eqn_le:PBC}, has a nontrivial null space.

By inspection it is obvious that for (PBC) $\alpha_0=1$ holds with $(\tilde U,\tilde V) := (1,1) \in \mathcal{N}(\mathcal{L}_1)$.
Straightforward computations show that the null space is one-dimensional. 

For (DBC) another straightforward, although a little longer computation gives that bifurcation values are of the form 
$\alpha = 1/|\cos b|$, where $b$ solves 
\begin{equation}\label{eqn_le:b}
  Cb + \tan b = 0 \,.
\end{equation}
We denote by $b_0\in (\pi/2,\pi)$ the smallest positive solution, set $\alpha_0 = 1/|\cos b_0|$, and note that
\begin{equation}\label{eqn_le:alpha0}
  \begin{pmatrix} \tilde U\\ \tilde V\end{pmatrix} := \begin{pmatrix}\sin(b_0 X)\\ \sin(b_0(1-X))\end{pmatrix} 
  \in \mathcal{N}(\mathcal{L}_{\alpha_0}) \,.
\end{equation}
Note that $\tilde U>0$ in $(0,1]$ and $\tilde V > 0$ in $[0,1)$. Again the null space is one-dimensional.

For both types of boundary conditions, infinite increasing sequences of bifurcation values exist. However, for all bifurcation values
larger than $\alpha_0$, the null spaces, and therefore the bifurcating solutions, consist of functions with changing signs, which are
irrelevant for our application and have no chance to be the long-time limit of non-negative solutions.

The normal form reduction close to the bifurcation point is derived by choosing values of $\alpha$ close to $\alpha_0$ and by 
making the ansatz 
\begin{equation}\label{eqn_le:ansatz}
   \begin{pmatrix}U(X,t)\\ V(X,t)\end{pmatrix} = (\alpha-\alpha_0)B(t) \begin{pmatrix}\tilde U(X)\\ \tilde V(X)\end{pmatrix} 
     + \mathcal{O}((\alpha-\alpha_0)^2) \,.
\end{equation}
For the linearized operator and its formal adjoint with respect to the scalar product in $L^2((0,1))^2$, the symmetry property
$$
  \mathcal{L}_\alpha^*(U,V) = \begin{pmatrix} \mathcal{L}_{\alpha,2} (V,U)\\\mathcal{L}_{\alpha,1} (V,U)\end{pmatrix}
$$
holds. Thus, the scalar product of \eqref{eqn_le:linearization}, written in the form
$$
  \partial_t \begin{pmatrix} U\\V\end{pmatrix} = \mathcal{L}_{\alpha_0} (U,V) 
  + \left( \alpha-\alpha_0 - \frac{\alpha\beta(U+V)}{1+\beta(U+V)} \right)\begin{pmatrix} V\\U\end{pmatrix} \,,
$$
with $(\tilde V,\tilde U)$ gives
$$
  \frac{\dd}{\dd t} \int_0^1 (U\tilde V + V\tilde U)\dd X = \int_0^1 
  \left( \alpha-\alpha_0 - \frac{\alpha\beta(U+V)}{1+\beta(U+V)} \right)(V\tilde V + U\tilde U)\dd X \,.
$$
Substitution of the ansatz \eqref{eqn_le:ansatz} leads to
\begin{equation}\label{eqn_le:normal-form}
  \frac{\dd B}{\dd t} = (\alpha - \alpha_0) B(\kappa_1 - a\kappa_2) + \mathcal{O}((\alpha-\alpha_0)^2) \,,
\end{equation}
with 
$$
  \kappa_1 = \frac{\int_0^1(\tilde U^2 + \tilde V^2)\dd X}{2\int_0^1 \tilde U \tilde V \dd X} > 0 \,,\qquad
  \kappa_2 = \frac{\alpha_0\int_0^1\beta(\tilde U + \tilde V)(\tilde U^2 + \tilde V^2)\dd X}{2\int_0^1 \tilde U \tilde V \dd X} > 0 \,.
$$
This is the normal form of the transcritical bifurcation. It indicates that for $\alpha<\alpha_0$ the trivial steady state is stable,
whereas for $\alpha>\alpha_0$, stability is transferred to the bifurcating steady state
\begin{equation}\label{eqn_le:stat-expansion}
  \begin{pmatrix} \bar U\\ \bar V\end{pmatrix} 
  = (\alpha-\alpha_0)\frac{\kappa_1}{\kappa_2}\begin{pmatrix} \tilde U\\ \tilde V\end{pmatrix} + \mathcal{O}((\alpha-\alpha_0)^2) \,.
\end{equation}
This result is of course only formal. For the steady state bifurcation, however,
our computations, in particular $\kappa_1,\kappa_2\ne 0$, verify the conditions of the Crandall-Rabinowitz theory
\cite{Crandall1971} for bifurcations from a simple eigenvalue. Without going into further detail, we shall state the result below.
We do not attempt to rigorously justify the dynamic normal form reduction \eqref{eqn_le:normal-form}, since the stability of the 
bifurcating state will be proved in the following section, anyway.

\begin{myprop}
\label{prop_le:bifurcation}
Let \eqref{eqn_le:cbound} hold and let $\alpha_0$ be defined as above. 
Then there exists $\alpha_1>\alpha_0$, such that the system \eqref{eqn_le:main} with \eqref{eqn_le:DBC} or \eqref{eqn_le:PBC} has 
a smooth (with respect to $\alpha$) branch of nontrivial solutions 
$(\bar u, \bar v): [\alpha_0,\alpha_1)\to\mathcal{C}([0,1])^2$ 
with $(\bar u,\bar v)\bigm|_{\alpha=\alpha_0} = (0,0)$ and with 
$$
  \frac{\dd}{\dd\alpha}\begin{pmatrix}\bar u\\ \bar v\end{pmatrix}(x)\Bigm|_{\alpha=\alpha_0} 
  = \frac{\kappa_1 C}{\kappa_2 c} \begin{pmatrix} \tilde U,\tilde V \end{pmatrix} \left(C\int_0^x \frac{\dd y}{c(y)}\right) \,. 
$$
In a neighborhood of the point $(u,v,\alpha)=(0,0,\alpha_0)$ in $\mathcal{C}([0,1])^2\times \mathbb{R}$
no other stationary solutions besides the trivial solution and the solutions on the nontrivial branch exist.
\end{myprop}

\begin{mylem}
\label{lem_le:positivity}
Let the assumptions of Proposition \ref{prop_le:bifurcation} hold. Then for $\alpha-\alpha_0>0$ small enough,
\begin{eqnarray*}
  &&\mbox{for (PBC):}\quad \bar u,\bar v > 0 \quad\mbox{in } [0,1] \,,\\
  &&\mbox{for (DBC):}\quad \bar u> 0 \quad\mbox{in } (0,1],\quad\bar v > 0 \quad\mbox{in } [0,1)\,.
\end{eqnarray*}
\end{mylem}

\begin{proof}
Proposition \ref{prop_le:bifurcation} justifies \eqref{eqn_le:stat-expansion}, with $\mathcal{O}((\alpha-\alpha_0)^2)$ to be understood uniformly in 
$X\in [0,1]$. This immediately implies the result, except for the case of (DBC), where the behavior of $\bar U$ close to $X=0$
and of $\bar V$ close to $X=1$ has to be examined. However, 
$$
   \bar U(0) = 0 \,,\quad \frac{\dd \bar U}{\dd X}(0) = \frac{\alpha}{C}\bar V(0) >0 \,,\qquad 
   \bar V(1) = 0\,,\quad \frac{\dd \bar V}{\dd X}(1) = -\frac{\alpha}{C}\bar U(1) < 0 \,,
$$
implies positivity of $\bar U$ near $X=0$ and of $\bar V$ near $X=1$, completing the proof.
\end{proof}

\subsection*{Periodic boundary conditions -- almost constant lateral flow speed}

In the case of periodic boundary conditions and of a constant lateral flow speed, the bifurcating branch of nontrivial solutions
is given by 
\begin{align}\label{eqn_le:const_sol}
  \begin{pmatrix}\bar u(x)\\  \bar v(x)\end{pmatrix} = \frac{\alpha-1}{2}\begin{pmatrix}1\\1\end{pmatrix}  \,,\qquad \alpha>1 \,.
\end{align}
So it is explicit, homogeneous, and global. We make use of these properties to construct a global branch for almost constant
lateral flow speeds, satisfying
$$
  c(x) = c_0 + \e c_1(x) \,,
$$
with a constant $c_0>0$, with a smooth function $c_1(x)$, and with a small parameter $\e$. Substitution of the transformation
$$
  \begin{pmatrix}\bar u(x)\\  \bar v(x)\end{pmatrix} = \frac{\alpha-1}{2}\begin{pmatrix}1\\1\end{pmatrix} 
  + \e (\alpha-1) \begin{pmatrix} u_1(x)\\ v_1(x)\end{pmatrix} 
$$
in \eqref{eqn_le:main} gives
\begin{align} \label{eqn_le:c-appr-const}
  \frac{\dd}{\dd x} \begin{pmatrix}u_1\\ v_1\end{pmatrix}-M(\alpha)\begin{pmatrix}u_1\\ v_1\end{pmatrix} 
   = h(x) + \e r(u_1,v_1,x;\alpha,\e)\,,
\end{align}
with
\begin{equation}
M(\alpha) =\frac{1}{2 c_0\alpha}\begin{pmatrix} 1-3\alpha && \alpha+1 \\ -\alpha-1 && 3\alpha-1 \end{pmatrix} \,, \qquad h(x)=-\frac{c_1'(x)}{2c_0}\begin{pmatrix} 1 \\ 1 \end{pmatrix} \,,\label{eqn_le:Mh}
\end{equation}
and where $r$ is smooth in all its variables and uniformly bounded in terms of $\alpha \to\infty$ and $\e\to 0$ with
commuting limits. It also satisfies 
\begin{equation} \label{eqn_le:r}
  r(\alpha=1) = \tilde r(u_1,v_1,x;\e)\begin{pmatrix}1\\1\end{pmatrix} \,.
\end{equation}
These observations are the result of a lengthy but straightforward computation. The next step is to rewrite \eqref{eqn_le:c-appr-const}
as a fixed point problem by inverting the linear operator on the left hand side on the space of periodic functions.
An explicit computation gives
\begin{equation}\label{eqn_le:u1v1}
\begin{pmatrix} u_1(x;\alpha,\e) \\ v_1(x;\alpha,\e)\end{pmatrix} =
e^{Mx} \left(\left(e^{-M}-\mathbb{I}\right)^{-1} \int_0^1 e^{-My}(h+\e r)(y)  \dd y+\int_0^x e^{-My}(h+\e r)(y)\dd y\right) \,,
\end{equation}
requiring the existence of the inverse of $e^{-M}-\mathbb{I}$, equivalent to the invertibility of $M$, which is true for
$\alpha>1$. Since $M(\alpha)$ also converges to an invertible matrix as $\alpha\to\infty$, the linear operator applied to $h+\e r$
on the right hand side is uniformly bounded in $\alpha\in [\alpha^*,\infty)$ with $\alpha^*>1$. Obviously, for $\e$ small enough,
solutions can be constructed by contraction on the space $\mathcal{C}_B(\mathbb{T}^1\times[\alpha^*,\infty))^2$ of bounded
continuous functions
of the periodic variable $x$ and of the parameter $\alpha$. We skip the details of the proof.

\begin{myprop}
Let $\alpha^*>1$ and let
\begin{align*}
c(x)=c_0+\e c_1(x) \,,\qquad\mbox{with } c_0>0 \,,\quad c_1 \in \mathcal{C}^1(\mathbb{T}^1) \,.
\end{align*}
Then there exist $\e_0,C>0$ such that for every $\e\le \e_0$ and for every $\alpha\ge\alpha^*$, the problem \eqref{eqn_le:main},
\eqref{eqn_le:PBC} with (PBC) has a smooth stationary solution $(\bar u,\bar v)$, satisfying
\begin{align*}
\|\bar u-(\alpha-1)(1/2 + \e u_1(\cdot\,;\alpha,0))\|_{L^\infty([0,1])} +
\|\bar v-(\alpha-1)(1/2 + \e v_1(\cdot\,;\alpha,0))\|_{L^\infty([0,1])}\leq \e^2 \alpha C \,,
\end{align*}
where $(u_1(x;\alpha,0), v_1(x;\alpha,0))\in \mathcal{C}_B(\mathbb{T}^1\times[\alpha^*,\infty))^2$ is given by \eqref{eqn_le:u1v1} with $\e=0$.
\end{myprop}

\begin{myrem}
Note the uniformity of the result in terms of $\alpha\in[\alpha^*,\infty)$. In particular the upper bound $\e_0$ for $\e$ and
the constant $C$ in the error estimate are independent from $\alpha\to\infty$. There is some subtility to the situation as
$\alpha\to 1$, since the matrix $(e^{-M(\alpha)} - \mathbb{I})^{-1}$ blows up in this limit. However the limiting right hand sides
being proportional to the vector $(1,1)$ (see \eqref{eqn_le:Mh}, \eqref{eqn_le:r}) satisfy the solvability conditions for $\alpha=1$, so that
a bounded passage to the limit can be expected. We do not carry out this limit in detail, since it is roughly equivalent to the 
bifurcation analysis above.
\end{myrem}

We conclude with an example: Setting 
\begin{equation}\label{eqn_le:c-ex}
  c(x)=1+\frac{3\e}{4} \cos(2\pi x)
\end{equation} 
yields
\begin{align}
\label{eqn_le:PBC_almost}
& u_1(x;\alpha,0)=-\frac{3\pi\left(2\alpha\pi \cos(2\pi x)\right)-(\alpha-1)\sin(2\pi x)}{8 \left(\alpha-1+2\pi^2\alpha\right)}\,,\\
&v_1(x;\alpha,0)=-\frac{3\pi\left(2\alpha\pi \cos(2\pi x)\right)+(\alpha-1)\sin(2\pi x)}{8 \left(\alpha-1+2\pi^2\alpha\right)} \,.\nonumber
\end{align}
Figure \ref{fig_le:PBC_almost} shows the approximate steady state solution 
$(\alpha-1)(1/2 + \e u_1(x;\alpha,0), 1/2 + \e v_1(x;\alpha,0))$ together with a numerical solution computed as a
steady state of the time dependent problem. The lateral flow in the example corresponds qualitatively to a simulation with the full
Filament Based Lamellipodium Model \cite{MOSS}, where the cell is moving and the front is located at $x=0$ (identified with $x=1$) and the back at $x=0.5$. The lateral flow speed $c(x)$ has its minimum at the back and its maximum at the front. The steady state distributions of both filament families have their maximum at the back, due to the accumulation of filaments by the lateral flow. 
The maximum of the right-moving family is shifted slightly to the left as compared to the cell rear and vice versa for the left-moving filaments. The reason for this seemingly counter-intuitive result is that filaments of the right-moving family are produced by left-moving filaments, which shifts the maximum of the right-moving family to the left.

\begin{figure}[h]
\caption{Approximate and numerical solution for $\alpha=10$ and $\e=0.1$ with $c(x)$ given by \eqref{eqn_le:c-ex}. The end density of the right moving filaments $u$ is depicted in blue (dashed) and that of the left moving filaments $v$ (solid) in red. Thin lines are the
asymptotic approximations, thick lines are the numerical solution of the time dependent problem after $t=20$ and the left $y$-axis applies. In green (dotted) the lateral flow $c(x)$ is depicted, the values are in relation to the right $y$-axis.}
  \centering
\includegraphics[width=0.8\textwidth]{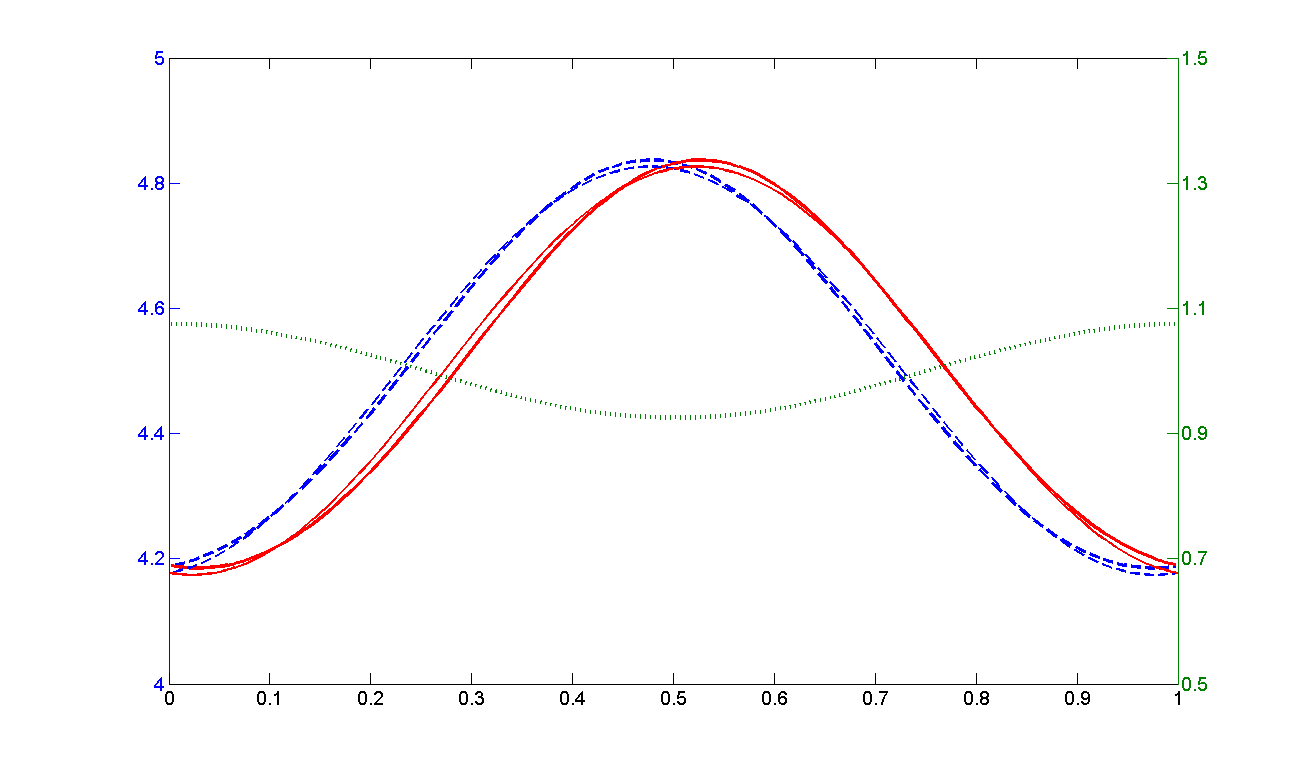}
\label{fig_le:PBC_almost}
\end{figure}

\subsection*{Dirichlet boundary conditions -- constant lateral flow speed}

In terms of the new unknowns $p=\bar u+\bar v$, $q=\bar u-\bar v$, the stationary version of \eqref{eqn_le:main} with constant lateral flow speed $c$ can be written as
\begin{equation}\label{eqn_le:pq}
  c p^\prime = -q\left( 1+ \frac{\alpha}{1+p}\right) \,,\qquad cq^\prime = p\left( \frac{\alpha}{1+p} - 1\right) \,.
\end{equation}
The boundary conditions (DBC) translate to
$$
  p(0)+q(0) = 0 \,,\qquad p(1)-q(1) = 0\,.
$$
The bifurcating solutions constructed above have the symmetry
$$
   \bar u(x) = \bar v(1-x) \qquad\longleftrightarrow\qquad p(x) = p(1-x) \,,\quad q(x) = -q(1-x) \,,
$$
and we shall look for solutions with this property. This allows to reduce the problem to the interval $[0,1/2]$ with the
boundary conditions
$$
  p(0)+q(0) = 0 \,,\qquad q(1/2) = 0\,.
$$
Viewing \eqref{eqn_le:pq} as a dynamical system and assuming $\alpha>1$, it has the critical points $(p,q)=(0,0)$, which is a center,
and the saddle $(p,q)=(\alpha-1,0)$. We look for solutions following trajectories in the region bounded by $p+q=0$, by $q=0$, 
and by the stable manifold of the saddle. The system has the first integral
\begin{equation}\label{eqn_le:1st-int}
  E_0 = q^2 + E(p,\alpha) \,,\qquad\mbox{with } E(p,\alpha) = - p^2 + 4\alpha p - 4\alpha(\alpha+1)\ln\left( 1 + \frac{p}{\alpha+1}\right) \,,
\end{equation}
which has been used for drawing the trajectories in Fig. \ref{fig_le:phaseportrait}. 

\begin{figure}[h]
\caption{This figure shows the trajectories in $(p,q)$-phase space of the stationary equation with constant $c$ and $(DBC)$. The orange, diagonal lines represent the boundary conditions $q=-p$ at $x=0$ and $q=p$ at $x=1$. Dashed curves represent \eqref{eqn_le:1st-int} for different values of $E_0$ and the blue lines the corresponding solutions curves fulfilling $p(0)+q(0)=0$. The red line is desired the solution satisfying both boundary conditions.}
  \centering
\includegraphics[width=0.6\textwidth]{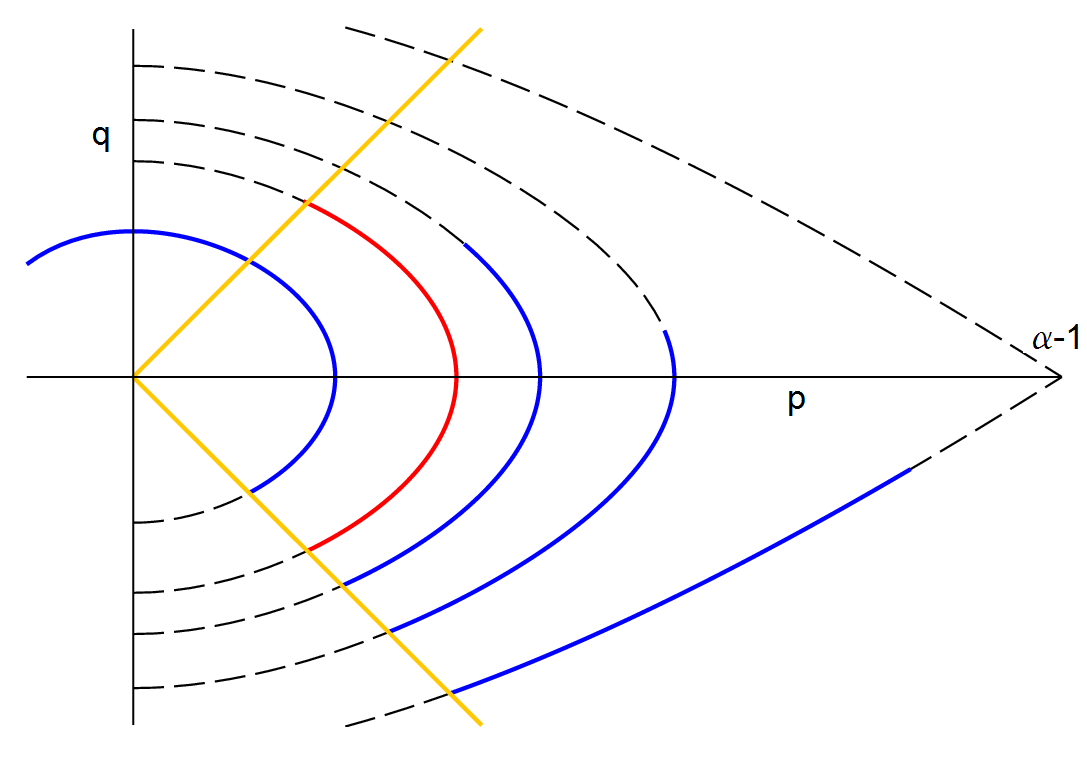}
\label{fig_le:phaseportrait}
\end{figure}

The above mentioned region corresponds to 
$$
E_0\in [0,E^*(\alpha)] \,,\qquad \mbox{with }  E^*(\alpha) = E(\alpha-1,\alpha) = (3\alpha+1)(\alpha-1) 
- 4\alpha(\alpha+1)\ln\left(\frac{2\alpha}{\alpha+1}\right) \,,
$$
where $E_0=0$ corresponds to the origin and $E_0=E^*(\alpha)$ to the stable manifold of the saddle. Among the trajectories 
connecting the segments $p+q=0$ and $q=0$, we need one which takes 'time' $1/2$.
To find an appropriate value of $E_0$, we compute $q<0$ from \eqref{eqn_le:1st-int}, substitute it in the first equation in \eqref{eqn_le:pq},
and integrate. This produces an equation for $E_0$:
\begin{equation}\label{eqn_le:I}
  I(E_0,\alpha) = c\int_{p_0(E_0,\alpha)}^{p_1(E_0,\alpha)} \frac{\dd p}{(1+\alpha/(1+p))\sqrt{E_0 - E(p,\alpha)}} = \frac{1}{2} \,,
\end{equation}
where $p(0)=p_0(E_0,\alpha)$ and $p(1/2)=p_1(E_0,\alpha)$ are the unique solutions of
$$
   E_0 = E(p_0,\alpha) + p_0^2 \,,\qquad E_0 = E(p_1,\alpha) \,.
$$
Information about the range of $I(\cdot,\alpha)$ is needed. For studying the limit as $E_0\to 0$, the Taylor expansion
$E(p,\alpha) = \frac{\alpha-1}{\alpha+1}p^2 + \mathcal{O}(p^3)$ and its consequence $p_0\approx \sqrt{E_0(\alpha+1)/(2\alpha)}$,
$p_1\approx \sqrt{E_0(\alpha+1)/(\alpha-1)}$, imply
$$
  I(0,\alpha) = \frac{c}{\sqrt{\alpha^2 - 1}} \left( \frac{\pi}{2} - \arcsin\sqrt{\frac{\alpha - 1}{2\alpha}}\right) \,.
$$
As a consistency check, it is easily verified that the unique solution of the equation $I(0,\alpha)=1/2$ is the bifurcation value
$\alpha_0$ defined by \eqref{eqn_le:b}, \eqref{eqn_le:alpha0}. The uniqueness follows since $I(0,\alpha)$ is strictly decreasing as 
a function of $\alpha$, i.e. 
$$
  I(0,\alpha) < 1/2 \qquad \mbox{iff } \alpha > \alpha_0 \,.
$$
On the other hand, by $p_1(E^*(\alpha),\alpha) = \alpha-1$ and $\partial_p E(\alpha-1,\alpha)$ the integrand in \eqref{eqn_le:I}
looses its integrability as $E_0\to E^*(\alpha)$, implying
$$
  \lim_{E_0\to E^*(\alpha)}  I(E_0,\alpha) = \infty \,,
$$
which is actually obvious, since along the stable manifold of the saddle the critical point and (therefore the line $q=0$) cannot be
reached in finite 'time'. These results and the continuity of $I(E_0,\alpha)$ complete the proof of the following proposition.

\begin{myprop}
\label{thm_le:solno_alt}
Let $\alpha>\alpha_0$ as defined by \eqref{eqn_le:b}, \eqref{eqn_le:alpha0} and let $c(x)=\,$const. Then the equation \eqref{eqn_le:main} together with $(DBC)$ has a non-trivial stationary solution.
\end{myprop}

\begin{myrem}
Uniqueness of the steady state is equivalent to strict monotonicity of $I(E_0,\alpha)$ as a function of $E_0$.
This will be proved indirectly by the decay result of the next section.
\end{myrem}

\section{Stability of Nontrivial Steady States}
\label{sec_le:lyapunov}
 
The main result of this section is that whenever a nontrivial steady state, as constructed (for certain cases) in the previous section,
exists then it attracts all nontrivial solutions. We therefore make the assumption that there exists a stationary solution 
$(\bar u(x),\bar v(x))$ of \eqref{eqn_le:main}, satisfying
\begin{eqnarray}
   \mbox{either}&&\mbox{(PBC),} \quad (\bar u,\bar v) \in \mathcal{C}^1(\mathbb{T}^1)^2 \,,\quad \mbox{and}\quad
      m \le\bar u(x),\bar v(x) \le M \,,\label{eqn_le:PBC-ass}\\
   \mbox{or}&&\mbox{(DBC),} \quad (\bar u,\bar v) \in \mathcal{C}^1([0,1])^2 \,,\quad \mbox{and}\quad
     m \le \frac{\bar u(x)}{x}, \frac{\bar v(x)}{1-x} \le M\,,\label{eqn_le:DBC-ass}
\end{eqnarray}
for $x\in[0,1]$, with positive constants $m,M$.

For proving exponential convergence to nontrivial steady states, we need to strengthen the result of Proposition \ref{prop_le:positivity}.
Note that the following result is sharp in terms of the values of the parameter $\alpha$, proving uniform-in-time bounds away
from zero for all parameter values above the bifurcation value, where the zero solution looses its stability.

\begin{mylem}
\label{lem_le:positivity}
Let the assumptions of Proposition \ref{prop_le:positivity} hold and let $\alpha>\alpha_0$, where $\alpha_0$ is the bifurcation value
as defined in the preceding section. Then for (PBC) there exist $T,m>0$ such that 
$$
u(x,t), v(x,t)\geq m \qquad \text{for}\ (x,t)\in[0,1]\times[T,\infty)\,.
$$
For (DBC) there exist $T,a>0$ such that 
$$
u(x,t) \ge a\sin\left(b_0C\int_0^x \frac{\dd y}{c(y)}\right),\ v(x,t)\geq a\sin\left(b_0C\int_x^1 \frac{\dd y}{c(y)}\right),
   \ \text{for } (x,t)\in[0,1]\times[T,\infty)\,.
$$
Furthermore, again for (DBC), there exists a constant $M>0$ such that
$$
  u(x,t) \le M x \,,\quad v(x,t) \le M(1-x) \,,\qquad\text{for } (x,t)\in[0,1]\times[T,\infty)\,.
$$
\end{mylem}

\begin{proof} (PBC): As a consequence of Proposition \ref{prop_le:positivity}, there exist $m_0, T>0$, such that 
$$
  u(x,T), v(x,T) \ge m_0 > 0 \qquad\mbox{for } x\in[0,1] \,.
$$
We choose 
$$
  M = \frac{\underline c}{C}\min\left\{ m_0, \frac{\alpha-1}{2} \right\} >0 \,,
$$
implying 
$$
  U(X,T), V(X,T) \ge M  \qquad\mbox{for } X\in[0,1] \,.
$$
Assume the bound holds for $V$ for later times and consider the equation along
characteristics for $U$:
$$
  \dot U = \frac{\alpha V}{1+\beta(U+V)} - U \ge \frac{\alpha M}{1 + \bar\beta (U+M)} - U \,.
$$
The right hand side is nonnegative for $U=M$ with the consequence 
$U\ge M$. Analogously for $V$. This completes the proof with $m = MC/\bar c$.\\
(DBC): By the result of Proposition \ref{prop_le:positivity}, there exist $M>0$, $X_0\in(0,1)$, $t_2>t_1\ge 0$, such that 
$V \ge M$ in $[0,X_0]\times[t_1,t_2]$. Along $U$-characteristics staying inside this rectangle,
$$
    \dot U\ge \frac{\alpha M}{1+ \bar\beta(U+M)} - U 
   \ge \frac{\alpha M}{2(1 + 2\bar\beta M)} \,,
$$
holds, as long as, furthermore, 
$$
  U \le \min\left\{ M, \frac{\alpha M}{2(1 + 2\bar\beta M)}\right\} \,.
$$
For characteristics starting at $X=0$, $t=t^*\in [t_1,t_2)$, this implies
$$
  U \ge \frac{\alpha M}{2(1 + 2\bar\beta M)} (t-t^*) = \frac{\alpha M}{2(1 + 2\bar\beta M)} \,\frac{X}{C} \,,
$$
until 
$$
 U = \min\left\{ M, \frac{\alpha M}{2(1 + 2\bar\beta M)}\min\left\{ 1,t_2-t^*, X_0/C \right\}\right\} \,.
$$
As a conclusion, there is a time $T\in (t_1,t_2]$ such that $U(X,T)$ (and therefore also $u(x,T)$) increases (as a 
function of $X$) at least linearly away from $X=0$ up to a certain point, after which, by Proposition \ref{prop_le:positivity}, it is 
bounded from below by a positive constant. This implies the existence of $A_0>0$, such that 
$$
  U(X,T) \ge A_0 \tilde U(X) \,,\qquad V(X,T) \ge A_0 \tilde V(X)\,,
$$
with $\tilde U, \tilde V$ defined by \eqref{eqn_le:b}, \eqref{eqn_le:alpha0}, and the second inequality is proven analogously. Now we choose 
$$
  A := \min\left\{ A_0, \frac{\alpha-\alpha_0}{2\alpha_0\bar\beta}\right\} \,,
$$
and assume $V(X,t)\ge A \tilde V(X)$ for $(X,t)\in [0,1]\times[T,\infty)$, in the equation for $U$. This implies
$$
  C\frac{\dd U}{\dd X} \ge  \frac{\alpha A\tilde V}{1 + \bar\beta (U+A\tilde V)} - U \,,\qquad t\ge T\,.
$$
We claim that $U(X) = A\tilde U(X)$ is a subsolution for $t\ge T$, which follows from 
$A\tilde U(X) \le U(X,T)$ and 
\begin{eqnarray*}
  C\frac{\dd (A\tilde U)}{\dd X} &=& CAb_0 \cos(b_0 X) = -A \tan b_0\cos(b_0 X) \\
  &=& \alpha_0 [A (\sin b_0\cos(b_0 X) - \sin(b_0 X)\cos b_0)] - A\sin(b_0 X) \\
  &=& \frac{\alpha A \tilde V}{1 + (\alpha/\alpha_0-1)} - A\tilde U
  \le \frac{\alpha A\tilde V}{1 + 2\bar\beta A} - A\tilde U 
  \le \frac{\alpha A\tilde V}{1 + \bar\beta(A\tilde U + A\tilde V)} - A\tilde U \,.
\end{eqnarray*}
On the other hand, it can be proved analogously that $U \ge A\tilde U$ implies $V \ge A\tilde V$. This completes
the proof of the lower bound with $a:= AC/\overline c$.

For the proof of the upper bound we observe that along characteristics 
$$
  C\frac{\dd U}{\dd X} \le \frac{\alpha}{\underline\beta} \,,
$$
implying at most linear growth of $U$ in terms of $X$, and therefore also of $u$ in terms of $x$, with the analogous argument
for $v$.
\end{proof}

The convergence analysis is based on the Lyapunov functional
\begin{equation}
  \mathcal{H}[u,v] := \frac{1}{2}\int_0^1 c\left(\frac{\bar v}{\bar u}(u-\bar u)^2
    +\frac{\bar u}{\bar v}(v-\bar v)^2\right)\dd x \,.\label{eqn_le:lya}
\end{equation}
Note that under the assumptions of Lemma \ref{lem_le:positivity} and with \eqref{eqn_le:PBC-ass} or \eqref{eqn_le:DBC-ass}, the Lyapunov functional
is well defined along the solution of \eqref{eqn_le:main} for $t\ge T$, since for (DBC) the integrand can be continuously extended 
by the value zero to $x=0,1$. For the computation of the time derivative of $\mathcal{H}[u,v]$ along solutions, the computation
\begin{align*}
-\int_0^1 c\frac{\bar v}{\bar u}&\left(u-\bar u\right)\partial_x(c u)\dd x\\
&=-\int_0^1 \frac{c\bar v}{c\bar u} \partial_x\frac{(c u-c \bar u)^2}{2}\dd x
   -\int_0^1 c\frac{\bar v}{\bar u}\left(u-\bar u\right)\partial_x(c \bar u)\dd x\\
&= \int_0^1 \frac{(c u-c\bar u)^2}{2} \,\frac{\bar u \partial_x(c\bar v) - \bar v \partial_x(c\bar u)}{c \bar u^2}\dd x 
-\int_0^1 c\frac{\bar v}{\bar u}\left(u-\bar u\right)\partial_x(c \bar u)\dd x\\
&=\int c \frac{(u-\bar u)^2}{2\bar u^2} \left(\bar u \partial_x(c\bar v) - \bar v \partial_x(c\bar u)\right)\dd x  
   -\int_0^1 c\,\frac{\bar v}{ \bar u}\left( u-\bar u\right)\partial_x(c \bar u)\dd x
\end{align*}
will be used, where the derivatives on the right hand side can be eliminated by using the stationary equations. 
The integration by parts leading to the third line does not produce any boundary terms since $\bar v(u-\bar u)^2/\bar u$ is
periodic for (PBC) and the continuous extension vanishes at $x=0,1$ for (DBC). The analogous computation for the second part
of $\mathcal{H}[u,v]$ leads to 
$$
  \frac{\dd}{\dd t}\mathcal{H}[u,v]
    = -\frac{\alpha}{2}\int_0^1 \frac{c J(u,v,\bar u,\bar v)}{(1+u+v)(1+\bar u+\bar v)}\dd x\,,
$$
with 
\begin{eqnarray}
\label{eqn_le:J}
J(u,v,\bar u,\bar v)&:=&\left(\frac{u-\bar u}{\bar u}\right)^2 \left[(1+u+v)\left(\bar u^2+\bar v^2\right)+2 \bar v^2 \bar u\right]\\
&&+\left(\frac{v-\bar v}{\bar v}\right)^2 \left[(1+u+v)\left(\bar u^2+\bar v^2\right)+2 \bar u^2 \bar v\right]\nonumber\\
&&-2\frac{u-\bar u}{\bar u}\,\frac{v-\bar v}{\bar v} \left(\bar u^2+\bar v^2+\bar u^2\bar v+ \bar u\bar v^2\right) \,.\nonumber
\end{eqnarray}
Nonnegativity of $J$ is not obvious, but actually even a coercivity property can be shown:

\begin{mythm}\label{thm_le:convergence}
Let the assumptions of Proposition \ref{prop_le:positivity} hold, let $\alpha>\alpha_0$ with the bifurcation value $\alpha_0$, let
a stationary solution $(\bar u,\bar v)$ of \eqref{eqn_le:main} exist, which satisfies either \eqref{eqn_le:PBC-ass} or \eqref{eqn_le:DBC-ass},
and let $\mathcal{H}[u,v]$ be defined by \eqref{eqn_le:lya}. Then there exists a constant $\gamma>0$ such that
$$
  \mathcal{H}[u,v](t) \le e^{-\gamma(t-T)}\mathcal{H}[u,v](T) \,,\qquad\mbox{for } t\ge T \,,
$$
with $T$ from Lemma \ref{lem_le:positivity}.
\end{mythm}

\begin{proof}
The representation \eqref{eqn_le:J} of $J$ suggests the notation
$$
  \hat u := \frac{u-\bar u}{\bar u} \,,\qquad \hat v := \frac{v-\bar v}{\bar v} \,,
$$
with $(\hat u,\hat v)\in [-1,\infty)^2$. This domain can be split into three subdomains, defined by $\hat u\hat v \le 0$, 
$0 < \hat u/\hat v \le 1$, and, respectively, $0 < \hat v/\hat u <1$.
 
For $\hat u\hat v \le 0$, $J$ is obviously nonnegative and satisfies
\begin{eqnarray*}
  J &\ge& \left( \bar u^2 + \bar v^2\right)\left( \hat u^2 + \hat v^2\right) 
    = \frac{\bar u^2 + \bar v^2}{\bar u \bar v}\left( \frac{\bar v}{\bar u}(u-\bar u)^2 + \frac{\bar u}{\bar v}(v-\bar v)^2\right) \\
    &\ge& 2\left( \frac{\bar v}{\bar u}(u-\bar u)^2 + \frac{\bar u}{\bar v}(v-\bar v)^2\right) \,.
\end{eqnarray*}
Now we consider the second case $0 < \hat u/\hat v \le 1$. The result for the third case then follows by symmetry.
With $\lambda = \hat u/\hat v$, we introduce another change of variables 
$(\hat u,\hat v) \to (\lambda,\hat v)\in (0,1]\times[-1,\infty)$, giving
\begin{eqnarray*}
 \frac{J}{\hat v^2} &=& \lambda^2 \left[ (1+\bar u+\bar v + (\bar u\lambda +\bar v)\hat v)\left(\bar u^2 + \bar v^2\right) 
    + 2\bar u \bar v^2\right] \\
  && + (1+\bar u+\bar v + (\bar u\lambda +\bar v)\hat v)\left(\bar u^2 + \bar v^2\right) 
    + 2\bar u^2 \bar v \\
  && - 2\lambda\left(\bar u^2+\bar v^2+\bar u^2\bar v+ \bar u\bar v^2\right) \\
  &=& \left(\bar u^2 + \bar v^2\right)(\bar u\lambda +\bar v)\left(\lambda^2 + 1 \right)(\hat v + 1) \\
  && + (1-\lambda)\left[ (1-\lambda)\left(\bar u^2 + \bar v^2\right) + \bar u^3(\lambda^2+1) 
    + (1-\lambda)^2\bar u \bar v^2 + 2\bar u^2 \bar v\right] \\
  &\ge& \left(\bar u^2 + \bar v^2\right) \left( \bar v(\hat v + 1) + (1-\lambda)^2\right) \,.
\end{eqnarray*}
With one more set of variables, $\tilde u = \hat u\sqrt{\bar u\bar v} = \sqrt{\bar v/\bar u}(u-\bar u)$,
$\tilde v = \hat v\sqrt{\bar u\bar v} = \sqrt{\bar u/\bar v}(v-\bar v)$, the above inequality is equivalent to
$$
  J\ge \frac{\bar u^2 + \bar v^2}{\bar u \bar v} \left( v \tilde v^2 + (\tilde u - \tilde v)^2\right)
   \ge 2(\tilde u - \tilde v)^2 + \kappa\tilde v^2 \,,\qquad\mbox{with } 
   \kappa = \inf_{[0,1]} \frac{\left(\bar u^2 + \bar v^2\right)v}{\bar u\bar v} \,.
$$
The results of Lemma \ref{lem_le:positivity} and the assumption \eqref{eqn_le:PBC-ass} or \eqref{eqn_le:DBC-ass} imply that $\bar u^2 + \bar v^2$
and $v/\bar v$ are bounded from below and $\bar u$ is bounded from above, with the consequence $\kappa>0$, and therefore
$$
  J\ge \frac{2\kappa}{4+\kappa} \left(\tilde u^2 + \tilde v^2\right) 
  = \frac{2\kappa}{4+\kappa}\left( \frac{\bar v}{\bar u}(u-\bar u)^2 + \frac{\bar u}{\bar v}(v-\bar v)^2\right) \,.
$$
With the common upper bound $M$ for $u,v,\bar u$, and $\bar v$, the proof is completed with 
$$
  \gamma = \frac{2\alpha\kappa}{(1+2M)^2(4+\kappa)} \,.
$$
\end{proof}

\bibliographystyle{plain}
\bibliography{bifpaper_bibliography}

\end{document}